\documentclass[12pt]{amsart}

\usepackage{amsmath}
\usepackage{amsthm}
\usepackage{amsfonts}
\usepackage{amssymb,amsmath,latexsym}
\usepackage{graphicx}
\usepackage{cite}

\theoremstyle{plain}
\newtheorem{theorem}{Theorem}
\newtheorem{lemma}[theorem]{Lemma}
\newtheorem{proposition}[theorem]{Proposition}
\newtheorem{corollary}[theorem]{Corollary}

\theoremstyle{definition}
\newtheorem{remark}[theorem]{Remark}

\title[Werner measure from chordal $\text{SLE}_{8/3}$]{A simple construction of Werner measure from chordal $\text{SLE}_{8/3}$}

\author{Robert~O. Bauer}
\address{Altgeld Hall\\Department of Mathematics\\ 	University of Illinois at Urbana-Champaign\\ 	1409 West Green Street \\ Urbana, IL 61801, USA}\email{rbauer@math.uiuc.edu}



\thanks{Research supported in part by NSF grant DMS 0604216}

\begin{document}

\maketitle

\begin{abstract}
We give a direct construction of the conformally invariant measure on self-avoiding loops in Riemann surfaces (Werner measure) from chordal $\text{SLE}_{8/3}$. We give a new proof of uniqueness of the measure and use Schramm's formula to construct a measure on boundary bubbles encircling an interior point. After establishing covariance properties for this bubble measure, we apply these properties to obtain a measure on loops by integrating measures on boundary bubbles. We calculate the distribution of the conformal radius of boundary bubbles encircling an interior point and deduce from it explicit upper and lower bounds for the loop measure.
\end{abstract}

\section{Introduction}
In \cite{wernerSAL}, Wendelin Werner established existence and uniqueness of a natural measure on the set of self-avoiding loops on Riemann surfaces. In this paper we call this measure {\em Werner measure}. The statement that it is a ``natural'' measure refers to how the measure on loops in a Riemann surface $S$ is related to the measure on loops in another surface $T$. If $\mu^S$ denotes the former measure and $\mu^T$ the latter, then the following is required:
\begin{itemize}
\item {\em (Conformal invariance)} If $S$ and $T$ are conformally equivalent and $\Phi$ is a conformal map from $S$ onto $T$,  then $\mu^T$ is the image measure of $\mu^S$ under $\Phi$.

\item {\em (Restriction)} If $S\subset T$, then $\mu^S$ is the restriction of $\mu^T$ to those loops in $T$ that stay in $S$.
\end{itemize}
The first requirement can be paraphrased as saying that the measures only depend on the conformal structure and that for each such structure there is essentially only one measure, just as conformally equivalent Riemann surfaces are essentially the same. The second requirement is a very simple consistency condition. It relates measures for certain surfaces that may or may not be conformally equivalent. An example for the latter is the case where $S$ is a ring domain in a torus $T$. In the construction of Werner measure we outline below, this will be instrumental in going from simple topologies to more complicated ones. The case where $S$ is conformally equivalent to $T$ is crucial for the uniqueness of Werner measure. To explain why, suppose  $S$ is properly contained in (i.e. not equal to) $T$. This can only occur if $S$ and $T$ are simply connected hyperbolic Riemann surfaces, for example if $T$ is the unit disk $\mathbb U=\{|z|<1\}$ and $S$ is a simply connected proper subdomain. Now there are two ways to obtain $\mu^S$ from $\mu^T$---one by conformal invariance, the other by restriction---and these have to agree, giving a condition on $\mu^T$. As $S$ ranges over all simply connected subdomains we get many such conditions and it is not at all obvious that even one measure exists that satisfies them all.

However, once one such measure is found, any multiple of this measure by a scalar $\lambda>0$ gives another such measure. Thus uniqueness refers here always to uniqueness up to a multiplicative constant. Also, for the measure to be of interest we need to exclude degenerate cases, where the measure is either identically zero or infinite on all sets of interest. A useful notion of nondegeneracy is the following: If $S$ is a proper simply connected subdomain of $\mathbb U$ containing $0$, and $\ell$ denotes a self-avoiding loop, then
\begin{equation}\label{E:nondegenrate}
0<\mu^{\mathbb U}(\ell\text{ surrounds $0$ but does not stay in $S$})<\infty.
\end{equation}
Restriction and conformal invariance then immediately imply that the total mass $|\mu^S|$ is infinite for any Riemann surface $S$. Indeed, by conformal invariance,
\begin{align}
&\mu^{\mathbb U}(\ell\text{ surrounds $0$ but does not stay in $\{|z|<2^{-1}\}$})\notag\\
&=\mu^{\{|z|<2^{-n}\}}(\ell\text{ surrounds $0$ but does not stay in $\{|z|<2^{-(n+1)}\}$}),\notag
\end{align}
for any $n\in\mathbb Z^+$, while, by restriction,
\begin{align}
&\mu^{\mathbb U}(\ell\text{ surrounds 0})\notag\\
&=\sum_{n=0}^{\infty}\mu^{\{|z|<2^{-n}\}}(\ell\text{ surrounds $0$ but does not stay in $\{|z|<2^{-(n+1)}\}$})\notag.
\end{align}
If $S$ is an arbitrary Riemann surface, then it contains the conformal image of $\mathbb U$. Whence, by conformal invariance and restriction,
\[
|\mu^S|\ge|\mu^{\mathbb U}|=\infty.
\]
We will later see that while $\mu^S$ is an infinite measure, it is $\sigma$-finite.

In this paper we give new proofs of uniqueness and existence of Werner measure and give explicit upper and lower bounds for the measure of loops that go around an annulus as a function of its modulus. More precisely, we do the following:

In Section \ref{S:P}, we introduce and define some of the objects we will use and show that given measures $\mu^{\langle U\rangle}$ on loops going around ring domains $U$ in a Riemann surface $S$, so that two such measures agree on their overlap, there exists a unique measure $\mu^S$ on loops in $S$ whose restriction to $\langle U\rangle$ agrees with $\mu^{\langle U\rangle}$, see Proposition \ref{P:restriction}. This is a straightforward application of direct sums of countably many measure spaces. That countably many spaces is enough will be shown to follow from the second countability of Riemann surfaces. This argument is not new and is briefly sketched in \cite{wernerSAL}. We spell it out here for the sake of completeness and the convenience of the reader.

We show in Section \ref{S:U} that---up to a multiplicative constant---there can be at most one measure on self-avoiding loops that is conformally invariant and satisfies restriction. This is a consequence of an explicit formula for
\begin{equation}\label{E:explicit}
\mu^{\mathbb U}(\ell\text{ surrounds $0$ but does not stay in $S$})
\end{equation}
for any simply connected domain $S\subset\mathbb U$ containing $0$. This formula was first derived in \cite[Proposition 3]{wernerSAL} using Loewner's theory of slit mappings. We give a new proof which is entirely elementary. First, we derive from the additivity of measures, conformal invariance, and restriction an explicit formula for
\[
\mu^{\mathbb U}(\ell\text{ surrounds $0$ but does not stay in $\{|z|<e^{-x}\}$}), \quad x>0.
\]
Then we use Taylor's formula to derive an expression for \eqref{E:explicit} from it.

In Section \ref{S:Bub}, we briefly introduce Schramm Loewner evolution (SLE) and show how chordal $\text{SLE}_{8/3}$ and Schramm's formula can be used to obtain a measure on {\em boundary bubbles encircling an interior point}. The basic idea is to condition chordal $\text{SLE}_{8/3}$ in the upper half-plane from $x$ to $\infty$ to pass to the right of $i$ and then let $x\to-\infty$. We show that the limiting object exists and that this measure, $Bub^{(i)}$, transforms like a quadratic differential, see Theorem \ref{T:bbubble}.

Boundary bubbles, which do not have to encircle a given interior point, have been introduced in \cite{lsw.chordal}, see also \cite[Section 5.5]{lawlerbook}. There, boundary bubbles arise as limits of outer boundaries of Brownian excursions as the excursion endpoints approach each other. The construction in these references leads to a $\sigma$-finite measure $Bub$ on boundary bubbles. $Bub^{(i)}$ then is the restriction of $Bub$ (or rather a suitable multiple thereof) to boundary bubbles encircling $i$. However we do not use this relationship in this paper. Instead, we exhibit a Bessel-type process $\theta_t$ on $[0,\pi]$ which serves as driving function for the Loewner equation. If $\theta_0>0$, we get chordal $\text{SLE}_{8/3}$ conditioned to pass to the right of $i$, if $\theta_0=0$, we get boundary bubbles encircling $i$.

The original construction of the conformally invariant restriction measure on self-avoiding loops in \cite{wernerSAL} is based on a conformally invariant measure on Brownian loops that had been introduced and studied in \cite{loopsoup}. By considering outer boundaries of Brownian loops, Werner obtains a measure on self-avoiding loops from it. However, the asymmetry of the construction, between the inside and the outside, requires an additional argument to check whether the measure on outer boundaries is invariant under inversion relative to an interior and an exterior point. For this part of the proof $\text{SLE}_{8/3}$-boundary bubbles are used in \cite{wernerSAL}.

We integrate over boundary bubbles encircling an interior point in Section \ref{S:E} to obtain a measure on loops. This idea of obtaining loop measures from boundary bubble measures goes back to \cite{loopsoup}, where it is applied to Brownian bubbles and Brownian loops.  Heuristically, the idea is the following: to count the loops going around the annulus $\{1<|z|<e^a\}$, we go along a cross-cut from the outer boundary component $\{|z|=e^a\}$ to the inner boundary component $\{|z|=1\}$, and count each loop when we come across it for the first time. When we encounter a loop for the first time, then the loop is a boundary bubble (encircling the point $0$) if we consider the part of the cross-cut we have already traversed as part of the boundary of a ring domain. In this way we can go from bubble measures to a loop measure.

To express this procedure as a Riemann sum it is necessary to determine what the ``increments'' along the cross-cut are. This boils down to parametrizing the cross-cut. Because the boundary bubble measures $Bub^{(i)}$, $Bub$ and also the measure on Brownian boundary bubbles all transform like a quadratic differential with respect to certain conformal transformations $\Phi$, the parametrization must be such that, infinitesimally, the change in parameter when the cross-cut is mapped by $\Phi$ is given by $dt=\Phi'(x)^2\ ds$. Here $x$ is the boundary point where the cross-cut emerges. This transformation-rule holds for any parametrization by ``conformal invariant,'' e.g. by conformal radius, conformal modulus (as in this paper), or half-plane capacity (as in \cite{wernerSAL} and \cite{loopsoup}). The combination of integrand ($Bub^{(i)}$) and integrator ($ds)$ then is conformally invariant.

A natural measure of the size of a boundary bubble encircling an interior point $p$ is the conformal radius of the interior of the bubble from $p$. In Section \ref{S:CR} we calculate the distribution of the conformal radius , which turns out to be a simple expression in the Dedekind $\eta$-function, see Theorem \ref{T:bubble}. Finally, we use the Koebe $1/4$ theorem to relate the conformal radius of a boundary bubble to the modulus of an annulus containing the bubble, in order to obtain upper and lower bounds for the measure of boundary bubbles going around an annulus as a function of the modulus, see Corollary \ref{C:bounds}.

\subsection{Further motivation}

As the outline above makes clear, this paper relies heavily on ideas developed in \cite{lsw.chordal}, \cite{loopsoup}, and \cite{wernerSAL}. A main motivation for us to write this paper, was to show in a way that is at once simple and self-contained that for $\text{SLE}_{8/3}$ ``chords'', ``bubbles,'' and ``loops'' are obtainable, one from the other, in a very intuitive way. Concerning simple loops in 2-dimensional geometries, Werner measure is a central object and it is desirable to have multiple approaches to its construction. One approach may lend itself to more easily investigate certain of Werner measure's aspects than another, and differing approaches may motivate different further questions. We close this section by mentioning two conjectures, one which situates Werner measure in a larger family of loop-measures, and the other which aims to describe Werner measure as a scaling limit.

In \cite{kontsevich.suhov}, Kontsevich and Suhov conjecture the existence of a 1-parameter family of locally conformally covariant measures on loops in Riemann surfaces with values in a certain determinant bundle. The parameter is the central charge $c$ from conformal field theory. For $c=0$, the bundle becomes trivial, and the measures are ordinary, scalar-valued measures. The parameter in the Schramm-Loewner evolution corresponding to central charge 0 is $\kappa=8/3$, and the scalar-valued measure is, in fact, given by Werner measure. It would be very interesting to construct the measures of Kontsevich and Suhov for other values of $c$. In this direction we note that our construction for boundary bubbles encircling an interior point also works for values of $\kappa\in(0,4]$. However, for $\kappa\neq8/3$ the resulting measures no longer transform like quadratic differentials so that our further construction of loop measures from bubble measures no longer applies.

Finally, both conformal field theory and Schramm Loewner evolution grew out of the desire to better describe and explain the behavior exhibited by 2-dimensional systems of statistical mechanics at criticality. Concerning the measure under consideration in this paper, it is conjectured to arise as the scaling limit of a certain model of random self-avoiding polygons on a regular lattice. To be specific, a self-avoiding polygon (SAP) of length $2n$ on the lattice $\mathbb Z^2$ is a finite sequence $\omega=(\omega_0,\omega_1,\dots,\omega_{2n})$ of points in $\mathbb Z^2$ with $|\omega_{k+1}-\omega_{k}|=1$, $0\le k<2n$, $\omega_0=\omega_{2n}$, and $\omega_j\neq\omega_k$ for $0\le j<k<2n$. We call a pair of consecutive points $\{\omega_j,\omega_{j+1}\}$ an edge of the polygon. It is known that the number of SAPs of length $2n$ with $\omega_0=0$ is of the order $\beta^{2n}$, where $\beta$ is the connective constant of the lattice. If we identify two SAPs $\omega,\omega'$ of length $2n$ if they have the same set of edges, then each equivalence class $[\omega]$ has $4n$ representatives. Define a measure $\mu_{SAP}$ on equivalence classes of SAPs by giving $[\omega]$ mass $\beta^{-2n}$. Then the following conjecture for a scaling limit of $\mu_{SAP}$ is given in \cite[Section 3.4.9]{lsw.SAW} (and also \cite[Section 7.1]{wernerSAL}): If $D$ is a planar domain and $N$ a positive integer, denote $\Gamma(D,N)$ the set of all $[\omega]$ so that $N^{-1}\omega$ is entirely (including lattice edges) contained in $D$. Then $\mu_{SAP}$ on $\Gamma(D,N)$ converges weakly (when the curves are rescaled by a factor $1/N$) to $\lambda\mu^D$, for some scalar $\lambda>0$. For example, for all $x>0$ we should have
\begin{align}
&\lambda\mu^D(\ell\text{ goes ``around'' }\{e^{-x}<|z|<1\})\notag\\
&=\lim_{N\to\infty}\mu_{SAP}([\omega]\text{ goes ``around'' }\{Ne^{-x}<|z|<N\}).
\end{align}

{\em Acknowledgement:} The author would like to thank Professor Roeckner and the SFB 701 at the University of Bielefeld for the hospitality during a visit where part of this work was completed.

\section{Preliminaries}\label{S:P}

A {\em Riemann surface} is a connected Hausdorff space $S$ together with a collection of charts $\{U_{\alpha},z_{\alpha}\}$ with the following properties: i) the $U_{\alpha}$ form an open covering of $S$, ii) each $z_{\alpha}$ is a homeomorphic mapping of $U_{\alpha}$ onto an open subset of the complex plane $\mathbb C$, iii) if $U_{\alpha}\cap U_{\beta}\neq\emptyset$, then $z_{\alpha\beta}=z_{\beta}\circ z_{\alpha}^{-1}$ is complex analytic on $z_{\alpha}(U_{\alpha}\cap U_{\beta})$.

A {\em loop} in a Riemann surface $S$ is a simple closed curve $\ell\subset S$. More precisely, a loop is a homeomorphism $\ell$ from the unit circle $S^1$ into $S$. We denote by $\mathcal L^S$ the set of loops in $S$.

An {\em ring domain} in $S$ is a Riemann surface $U\subset S$ whose boundary consists of two components and has genus zero (no handles). In particular, there exists a conformal map from $U$ onto some annulus $\{z:R_1<|z|<R_2\}$, where $0\le R_1<R_2\le\infty$. The number $a=\log R_2/R_1\in(0,\infty]$ is unique and is called the modulus of the ring domain $U$. We will write $\mod(U)=a$.

We say a loop $\ell$ separates the boundary components of a ring domain $U$ if $U\backslash\ell$ consists of two ring domains. We denote by $\langle U\rangle$ the set of all loops $\ell$ which separate the boundary components of $U$. Note that each loop $\ell$ in a Riemann surface $S$ separates the boundary components of some ring domain $U$ in $S$.

It is a classical result that Riemann surfaces are second countable, \cite{ahlfors.sario}. This has several important consequences. Let $S$ be a Riemann surface, $\rho$ a metric on $S$, and define a topology on $\mathcal L^S$ by taking the sets
\[
B(\ell,\epsilon)\equiv\{\ell'\in\mathcal L^S:\sup_{s\in S^1}\rho(\ell(s),\ell'(s))<\epsilon\}
\]
as a neighborhood basis at $\ell\in\mathcal L^S$ as $\epsilon$ varies over $(0,\infty)$.
Because $S^1$ is compact, this topology on $\mathcal L^S$ depends only on the topology of $S$ and not on the metric $\rho$.

\begin{lemma}\label{L:separable}
Let $S$ be a Riemann surface.

(i) There exists a countable collection of loops $\{\ell_n\}$ which is dense in $\mathcal L^S$.

(ii) There exists a countable collection of ring domains of finite modulus $\{U_n\}$ such that $\bigcup_n\langle U_n\rangle=\mathcal L^S$
.
\end{lemma}

\begin{proof}
Because Riemann surfaces are locally compact and second countable, it follows from \cite[Satz 31.5]{bauer:MI} that a Riemann surface $S$ is a Polish space. Recall that a topological space $S$ is a Polish space if the topology has a countable basis and if there exists a complete metric $\rho$ on $S$ which induces the topology. The path space of a Polish space is again a Polish space, \cite[Section 3.4]{stroock}. Since $\mathcal L^S$ is a subset of the path space of $S$, it follows that there exists a countable collection of loops $\{\ell_n\}$ which is dense in $\mathcal L^S$.

Consider now the countable collection of ring domains whose boundary consists of two of the loops from the collection $\{\ell_n\}$. We will show that this collection satisfies (ii): Since both boundary components of each such ring domain contains more than one point, the region has finite modulus. Next, if $\ell$ is a loop in $S$, then $\ell\in\langle U\rangle$ for some ring domain $U$. Denote $U_1$, $U_2$ the two components of $U\backslash\ell$. By (i), there exist loops $\ell',\ell''\in\{\ell_n\}$ such that $\ell'\in\langle U_1\rangle$, $\ell''\in\langle U_2\rangle$. Then $\ell',\ell''$ bound a ring domain $\tilde{U}$ and $\ell$ separates $\ell',\ell''$.
\end{proof}

From now on we will ignore the parametrization of a loop. Specifically,
we will identify two loops $\ell,\ell'$ if there exists a homeomorphism $\varphi:S^1\to S^1$ such that $\ell'=\ell\circ\varphi$.
We will use the same notation, $\ell$, for a loop and its equivalence class, and $\mathcal L^S$ for the set of equivalence classes.

If $U,V$ are ring domains in $S$ and $\langle U\rangle\cap\langle V\rangle\neq\emptyset$, then there exists a ring domain $W$ so that $\langle U\rangle\cap\langle V\rangle=\langle W\rangle$. In fact, $W$ is the connected component of $U\cap V$ whose boundary components are separated by a loop in $\langle U\rangle$ or $\langle V\rangle$. We will write $U\wedge V$ for this component. If, for convenience, we include the empty set as a ring domain and set $\langle\emptyset\rangle=\emptyset$, then
the collection of sets
\[
\{\langle U\rangle:\text{ $U$ ring domain in $S$}\}
\]
is stable under intersection, and $\langle U\rangle \cap\langle V\rangle=\langle U\wedge V\rangle$. We write $\sigma^S$ for the $\sigma$-algebra of subsets of $\mathcal L^S$ generated by this collection. If $U$ is a ring domain, then the collection of sets
\[
\{\langle V\rangle:\text{ $V$ ring domain in $U$ such that $\langle V\rangle\subset\langle U\rangle$}\}
\]
is also stable under intersection. We write $\sigma^{\langle U\rangle}$ for the $\sigma$-algebra of subsets of $\langle U\rangle$ generated by this collection. Note that the $\sigma$-algebra $\sigma^U$ is strictly bigger than $\sigma^{\langle U\rangle}$.

We recall that any $\sigma$-finite measure on $\sigma^S$ is uniquely determined by its values on $\{\langle U\rangle:\text{ $U$ ring domain in $S$}\}$.

\begin{proposition}\label{P:restriction}
Let $S$ be a Riemann surface. Suppose for each ring domain of finite modulus $U\subset S$ we are given a finite measure $\mu^{\langle U\rangle}$ on $(\langle U\rangle,\sigma^{\langle U\rangle})$ such that
\[
\mu^{\langle U\rangle}(\langle W\rangle)=\mu^{\langle V\rangle}(\langle W\rangle),
\]
whenever $U,V,W$ are ring domains with $W\subset U\cap V$. Then there exists a unique $\sigma$-finite measure $\mu^S$ on $(\mathcal L^S,\sigma^S)$ such that the restriction of $\mu^S$ to $\sigma^{\langle U\rangle}$ equals $\mu^{\langle U\rangle}$.
\end{proposition}

\begin{proof}
By the remark preceding the Theorem, the uniqueness of $\mu^S$ is not in doubt. To prove the existence of a measure $\mu^S$ with the desired properties,
let $\{U_n\}$ be a sequence of ring domains in $S$ with the properties from Lemma \ref{L:separable}. Define sets $\mathcal V_n, n\in\mathbb Z^+$ by $\mathcal V_1=\langle U_1\rangle$ and
\[
\mathcal V_n=\langle U_n\rangle\backslash(\langle U_1\rangle\cup\dots\cup\langle U_{n-1}\rangle),\quad n>1.
\]
Then the collection $\{\mathcal V_n\}$ is a partition of the loop-space $\mathcal L^S$. Define the trace $\sigma$-algebras
\[
\sigma_n=\{B\cap\mathcal V_n:B\in\sigma^{\langle U\rangle}\}.
\]
We now claim that for each $n\in\mathbb Z^+$, $\mathcal V_n\in\sigma^{\langle U_n\rangle}$, and that $\sigma^S$ consists of the sets $\bigcup_n V_n$, where $V_n\in\sigma_n$. The first claim is obviously true for $n=1$. For $n>1$, we simply note that in that case
\[
\mathcal V_n=\left(\cdots\left((\langle U_n\rangle\backslash\langle U_1\wedge U_n\rangle)\backslash\langle U_2\wedge U_n\rangle\right)\backslash\dots\right)\backslash\langle U_{n-1}\wedge U_n\rangle.
\]
Since $\langle U_n\wedge U_m\rangle\in\sigma^{\langle U_n\rangle}$ for any $n,m$, we get $\mathcal V_n\in\sigma^{\langle U_n\rangle}$. To prove the second claim, we note that the collection $\{\bigcup_n V_n: V_n\in\sigma_n\}$
is a $\sigma$-algebra. Furthermore, if $U$ is a ring domain in $S$, then
\[
\langle U\rangle=\bigcup_n (\langle U\rangle\cap\mathcal V_n)=\bigcup_n(\langle U\wedge U_n\rangle\cap\mathcal V_n).
\]
Since $\langle U\wedge U_n\rangle\in\sigma^{\langle U_n\rangle}$, the claim follows.

Thus, we can restrict $\mu^{\langle U_n\rangle}$ to $\sigma_n$ and  define a measure $\mu^S$ on $\sigma^S$ by
\begin{equation}\label{E:measure}
\mu^S(V)=\sum_n\mu^{\langle U_n\rangle}(V_n),
\end{equation}
where $V=\bigcup_n V_n$, and $ V_n\in\sigma_n$.

Finally, we show that this measure has the desired property, i.e. that the restriction of the measure $\mu^S$ defined in \eqref{E:measure} to $\sigma^{\langle U\rangle}$ equals $\mu^{\langle U\rangle}$. To this end, let $U$ be a ring domain in $S$. Then
\[
\langle U\rangle=\bigcup_n(\langle U\rangle\cap\mathcal V_n).
\]
It is easy to see that $\langle U\rangle\cap\mathcal V_n\in\sigma^{\langle U\rangle}\cap\sigma^{\langle U_n\rangle}$. Thus, by assumption,
\[
\mu^S(\langle U\rangle)=\sum_n\mu^{\langle U_n\rangle}(\langle U\rangle\cap\mathcal V_n)=\sum_n\mu^{\langle U\rangle}(\langle U\rangle\cap\mathcal V_n)=\mu^{\langle U\rangle}(\langle U\rangle).
\]
\end{proof}

Suppose now that for each Riemann surface $S$ we are given a nondegenerate measure $\mu^S$ on $(\mathcal L^S,\sigma^S)$. We say that the family $\{\mu^S\}$ is a  {\em conformal restriction family} if it satisfies the properties we listed at the beginning of the introduction:
\begin{itemize}
\item {\em conformal invariance.} If $\Phi$ is a conformal map from a Riemann surface $S$ onto another surface $T$, then $\Phi_*\mu^S=\mu^T$.
\item {\em restriction.} If $S$ is a Riemann surface contained in the surface $T$, then $\mu^T\upharpoonright_{\sigma^S}=\mu^S$.
\end{itemize}
Note that it is enough to check these properties on ring domains. I.e. conformal invariance is equivalent to the statement that if $U$ is a ring domain in $T$, then $\mu^S(\langle \Phi^{-1}(U)\rangle)=\mu^T(\langle U\rangle)$, while restriction is equivalent to the statement that if
$U$ is a ring domain in $S$, then $\mu^T(\langle U\rangle)=\mu^S(\langle U\rangle)$.

\section{Uniqueness}\label{S:U}

The following result, essentially Proposition 3 from \cite{wernerSAL}, shows that the combination of conformal invariance and restriction specifies the family. The proof in \cite{wernerSAL} proceeds by establishing a semi-group property and then bringing Loewner's theory of slit mappings to bear, in particular, that compositions of slit mappings are dense, in an appropriate sense, in the space of conformal maps. For an introduction to Loewner's method, see \cite{loewner} and Chapter 6 of \cite{ahlforsCI}. We will give a new proof of uniqueness, which is entirely elementary.

\begin{theorem}[Werner \cite{wernerSAL}]\label{T:uniqueness}
Up to a multiplicative constant, there is at most one conformal restriction family. In fact, if $\{\mu^S\}$ is a conformal restriction family, $D'\subset D$ simply connected planar domains not equal to $\mathbb C$, and $z\in D'$, then
\begin{equation}\label{E:phiprime}
\mu^{\mathbb C}(\langle D\backslash\{z\}\rangle\backslash\langle D'\backslash z\rangle)=c\log \Phi'(z),
\end{equation}
where $\Phi$ is the conformal map from $D'$ onto $D$ fixing $z$ and with positive derivative there.
\end{theorem}

\begin{proof}
Suppose $\{\mu^S\}$ is a conformal restriction family, and denote $\mu$ the restriction of the measure $\mu^{\mathbb C}$ to those loops that stay in the unit disk $\{|z|<1\}$ and surround $0$, i.e. $\mu=\mu^{\langle\{0<|z|<1\}\rangle}$. Then $\mu$ specifies the measure of $\langle U\rangle$ for each annulus $U=\{r<|z|<1\}$, where $0<r<1$. By conformal invariance and restriction, this determines all measures $\mu^S$, cf. Proposition \ref{P:restriction}.
Thus we need to show that $\mu$ is unique up to a multiplicative constant. To this end, for a loop $\ell\subset \mathbb C$, let $\ell_{\max}$ be the maximum modulus of points on the loop, and set
\[
p(x,y)=\mu(\ell_{\max}\in[e^{-y}, e^{-x})),
\]
where $0\le x\le y$. Note that, by definition of $\mu$, only loops which surround $0$ contribute to $p(x,y)$.  Then, by additivity of the measure $\mu$,
\[
p(x,y)+p(y,t)=p(x,t),
\]
whenever $0\le x\le y\le t$, and, by scale invariance of $\mu$,
\[
p(x,y)=p(0,y-x),
\]
for $0\le x\le y$. Furthermore, $p(0,0)=\mu(\emptyset)=0$, and, because $\mu$ is non-trivial, $p(0,x)\in(0,\infty)$ for $x\in(0,\infty)$. Thus, the function $x\in[0,\infty)\mapsto p(0,x)\in[0,\infty)$ is additive,
\[
p(0,x)+p(0,y)=p(0,x+y),
\]
and monotone,
\[
p(0,x)\le p(0,y),\quad \text{if }x\le y.
\]
Hence $p(0,x)=\lambda x$ with
\[
\lambda=\mu(\ell_{\max}\in[e^{-1},1)).
\]

Consider now a simply connected domain $D$ which is a subset of the unit disk and contains the point $0$. Then
\begin{align}\label{E:missD}
\mu(\ell\nsubseteq D)&=\lim_{\epsilon\searrow0}\mu(\ell_{\max}\in[\epsilon,1),\ell\nsubseteq D)\notag\\
&=\lim_{\epsilon\searrow0}(\mu(\ell_{\max}\in[\epsilon,1))-\mu(\ell\subset D,\ell_{\max}\in[\epsilon,1)).
\end{align}
Denote $\Phi_D$ the conformal map from $D$ onto the unit disk $\mathbb U=\{|z|<1\}$, normalized by $\Phi_D(0)=0$ and $\Phi_D'(0)>0$. By restriction and conformal invariance,
\[
\mu(\ell\subset D,\ell_{\max}\in[\epsilon,1))=\mu(\ell\subset\mathbb U,\ell\nsubseteq\Phi_D(\{|z|<\epsilon\})).
\]
If $m=\min\{|\Phi_D(z)|:|z|=\epsilon\}$, $M=\max\{|\Phi_D(z):|z|=\epsilon\}$, and $A\triangle B$ denotes the symmetric difference of two sets $A,B$, then
\begin{align}
&\{\ell\subset\mathbb U,\ell\nsubseteq\Phi_D(\{|z|<\epsilon\})\}\triangle\{\ell_{\max}\in[\Phi_D'(0)\epsilon,1)\}\notag\\
&=\{\ell_{\max}<\Phi_D'(0)\epsilon,\ell\nsubseteq\Phi_D(\{|z|<\epsilon\})\}\notag\\
&\qquad\cup
\{\ell_{\max}\ge\Phi_D'(0)\epsilon,\ell\subset\Phi_D(\{|z|<\epsilon\})\}\notag\\
&\subset\{\ell_{\max}\in[m,\Phi_D'(0)\epsilon)\}\cup\{\ell_{\max}\in[\Phi_D'(0)\epsilon,M)\}\notag\\
&=\{\ell_{\max}\in[m,M)\}.
\end{align}
By Taylor's theorem, both $M$ and $m$ equal $\Phi_D'(0)\epsilon+o(\epsilon)$. Using the scale invariance of $\mu$, it is now easy to see that
\begin{equation}\label{E:minmax}
\mu(\ell_{\max}\in[m,M))=o(1).
\end{equation}
From \eqref{E:missD} and \eqref{E:minmax}, we finally get
\begin{align}\label{E:lastlimit}
\mu(\ell\nsubseteq D)&=\lim_{\epsilon\searrow0}(\mu(\ell_{\max}\in[\epsilon,1))-\mu(\ell_{\max}\in[\Phi_D'(0)\epsilon,1))+o(1))\notag\\
&=\lambda\log\Phi_D'(0).
\end{align}
Equation \eqref{E:phiprime} then follows from conformal invariance.

It remains to show that knowing $\mu$ on the events $\{\ell\nsubseteq D\}$ for all simply connected subdomains of the unit disk containing  $0$ specifies $\mu(\ell\subset\{r<|z|<1\})$ for each $r\in(0,1)$. So, let $r\in(0,1)$ be given. Consider the set of Jordan arcs $a$ in the closed annulus $\{r\le|z|\le 1\}$ whose interior $\mathring{a}$ is contained in the open annulus $\{r<|z|<1\}$ and whose two endpoints lie on different boundary components of the annulus. We call such an arc a cross-cut. By a separability argument (like the one above for loops on Riemann surfaces) it is easy to see that there exists a countable collection of cross-cuts $\{a_k\}$ such that $\ell\in\langle\{r\le|z|<1\}\rangle$ if and only if $\ell\subset \mathbb U$ and the intersection of $\ell$ with $a_k$ is nonempty, for each $k$. Thus, by countable additivity of the measure $\mu$,
\[
\mu(\langle\{r\le|z|<1\}\rangle)=\lim_{N\to\infty}\mu(\bigcap_{n=1}^N\{\ell\cap a_n\neq\emptyset\}).
\]
By the inclusion/exclusion formula, the expression in the limit can be written using terms of the type $\mu(\bigcup_{k=1}^l\{\ell\cap a_{n_k}\neq\emptyset\})$, where $1\le l\le N$, $1\le n_k\le N$. Furthermore, if $D$ is the component of $\mathbb U\backslash(a_{n_1}\cup\dots\cup a_{n_l})$ which contains the point $0$, then
\[
\bigcup_{k=1}^l\{\ell\cap a_{n_k}\neq\emptyset\}=\{\ell\nsubseteq D\}.
\]
Finally, using for example the continuity of $x\mapsto p(0,x)$, it is easy to see that $\mu(\langle\{r\le|z|<1\})=\mu(\langle\{r<|z|<1\}\rangle)$, which concludes our argument.
\end{proof}

\begin{remark}
The expression for $\mu(\ell\nsubseteq D)$ as a limit is reminiscent of the definition of reduced extremal distance, see \cite[Section 4-14]{ahlforsCI}, and reduced modulus, see \cite[Section 3.2]{strebelQD}. In fact, our result shows that, up to a fixed multiplicative constant, $\mu(\ell\nsubseteq D)$ is equal to the reduced modulus of the set of loops which surround $0$, stay in $\mathbb U$, but not in $D$.
\end{remark}

\section{Boundary bubbles encircling an interior point}\label{S:Bub}

We consider chordal $\text{SLE}_{\kappa}$ in the $\mathbb H$ from $x\in\mathbb R$ to $\infty$. Following the notation in \cite{lawlerbook},
let $a=2/\kappa$ and for each $z\in\mathbb H$ denote $g_t(z)$ the solution to the {\em chordal Loewner equation}
\begin{equation}\label{E:CL}
\partial_t g_t(z)=\frac{a}{g_t(z)-U_t}, \quad g_0(z)=z,
\end{equation}
where $U_t$ is a linear Brownian motion with $\mathbb E[U_t^2]=t$. The solution exists up to time $T_z=\sup\{t:\min_{s\in[0,t]}|g_s(z)-U_s|>0\}$
and if $H_t=\{z:T_z>t\}$, then $g_t$ is the conformal map from $H_t$ onto $\mathbb H$ normalized by $g_t(z)=z+at/z+o(1/|z|)$, $z\to\infty$.
Furthermore, with probability 1, the set $K_t=\mathbb H\backslash H_t$ is generated by a curve $\gamma:[0,\infty]\to\overline{\mathbb H}$ in the
sense that $K_t$ is the complement of the unbounded component of $\mathbb H\backslash\gamma[0,t]$, see \cite{rohde.schramm}. We call the random curve $\gamma$ chordal
$\text{SLE}_{\kappa}$ in $\mathbb H$ from $x=U_0$ to $\infty$. If $\kappa\le 4$, then $\gamma$ is a.s. simple,
$\gamma[0,t]=K_t$, and $\gamma(0,\infty)\subset\mathbb H$.

Let $\kappa=8/3$, i.e $a=3/4$, and denote $P_x$ the
distribution of the random curve $\gamma$. By a result of Schramm, \cite{schramm.percolation},
\begin{equation}\label{E:Schramm}
P_x(\gamma\text{ passes right of $i$})=\frac{1}{2}\left(1+\frac{x}{\sqrt{1+x^2}}\right)\equiv q(x).
\end{equation}
It is well known that $\text{SLE}_{8/3}$ satisfies conformal restriction in the sense that if $D$ is a simply connected subdomain of $\mathbb H$ containing boundary neighborhoods of $x$ and $\infty$ in $\mathbb H$, then $P_x(\ \cdot\ |\gamma\subset D)=\Phi_D^* P_x$, where $\Phi_D$ is a conformal map from $D$ onto $\mathbb H$ fixing $x$ and $\infty$, and $\Phi_D^* P$ denoted the pull-back of the measure $P$ under $\Phi_D$. This property implies
\begin{equation}\label{E:cr}
P_x(\gamma\subset D)=\left(\Phi_D'(x)\Phi_D'(\infty)\right)^{5/8},
\end{equation}
where the derivative at $z=\infty$ is calculated relative to the local parameter $u=-1/z$, see \cite{lsw.chordal}. Furthermore, if $D,D'$ are ring domains in $\mathbb H$ of the same modulus containing boundary neighborhoods of $\infty$ in $\mathbb H$, $\Phi_{D,D'}$ the conformal map from $D$ onto $D'$ fixing $\infty$, $D$ containing a boundary neighborhood of $x$ in $\mathbb H$ and $D'$ containing a boundary neighborhood of $\Phi_{D,D'}(x)$ in $\mathbb H$, then
\begin{equation}\label{E:Beffara}
P_x(\gamma\in D)=\left(\Phi_{D,D'}'(x)\Phi_{D,D'}'(\infty)\right)^{5/8}P_{\Phi_{D,D'}(x)}(\gamma\in D').
\end{equation}
Equation \eqref{E:Beffara} was first established by an inclusion/exclusion argument in \cite{beffara}.

\begin{corollary}
If $D$ is a simply connected subdomain of $\mathbb H$ containing $i$ and a boundary neighborhood of $\infty$ in $\mathbb H$, and if $\Phi_D$ is the conformal map from $D$ onto $\mathbb H$ fixing $i$ and $\infty$, then
\begin{equation}\label{E:lim1}
\lim_{x\to-\infty} P_x(\gamma\subset D|\gamma\text{ passes right of $i$})=\Phi_D'(\infty)^2.
\end{equation}
If $D,D'\subset\mathbb H$ are ring domains of the same modulus each containing a boundary neighborhood of $\infty$ in $\mathbb H$ and so that $i\notin D, i\notin D'$, and if $\Phi_{D,D'}$ is the conformal map from $D$ onto $D'$ fixing $\infty$, then
\begin{equation}\label{E:lim2}
\frac{\lim_{x\to-\infty}P_x(\gamma\subset D|\gamma\text{ passes right of $i$})}{\lim_{x\to-\infty}P_x(\gamma\subset D'|\gamma\text{ passes right of $i$})}=\Phi_{D,D'}'(\infty)^2.
\end{equation}
\end{corollary}

\begin{proof}
Both identities follow readily from equations \eqref{E:Schramm}, \eqref{E:cr}, and \eqref{E:Beffara}. We sketch the argument for \eqref{E:lim2}. Using \eqref{E:Beffara} and an inclusion/exclusion argument it follows that
\begin{align}
&P_x(\gamma\subset D,\gamma\text{ passes right of $i$})\notag\\
&=\left(\Phi'(x)\Phi'(\infty)\right)^{5/8}P_{\Phi(x)}(\gamma\subset D',\gamma\text{ passes right of $i$}),\notag
\end{align}
where $\Phi=\Phi_{D,D'}$.
Whence, from \eqref{E:Schramm},
\begin{align}
&P_x(\gamma\subset D|\gamma\text{ passes right of $i$})\notag\\
&=\left(\Phi'(x)\Phi'(\infty)\right)^{5/8}P_{\Phi(x)}(\gamma\subset D'|\gamma\text{ passes right of $i$})\frac{q(\Phi(x))}{q(x)}.\notag
\end{align}
Equation \eqref{E:lim2} now follows from
\[
\lim_{x\to-\infty}\left(\Phi'(x)\Phi'(\infty)\right)^{5/8}=1,
\]
and
\[
\lim_{x\to-\infty}\frac{q(\Phi(x))}{q(x)}=\Phi'(\infty)^2.
\]
\end{proof}

Let $(g_t,t\ge0)$ be the solution to the chordal Loewner equation \eqref{E:CL} for $\text{SLE}_{\kappa}$ and $\kappa\le4$, i.e. $a\ge1/2$. With probability $1$, $i\notin\gamma(0,\infty)$. Thus $x_t=\Re g_t(i)$ and $y_t=\Im g_t(i)$ exist and we may define $f_t(z)$ by
\[
f_t(z)=\frac{g_t(z)-x_t}{y_t}.
\]
Then $f_t$ is the conformal map from $\mathbb H\backslash\gamma(0,t]$ onto $\mathbb H$ fixing $\infty$ and $i$. We also introduce $\theta_t\in[0,\pi]$ via
\[
\cot\theta_t=\frac{x_t-U_t}{y_t}.
\]
Then
\begin{align}
d x_t&=\frac{a}{y_t}\sin^2\theta_t\cot\theta_t\ dt,\notag\\
d y_t&=-\frac{a}{y_t}\sin^2\theta_t\ dt,\notag\\
d\theta_t&=(1-2a)\frac{\sin^4\theta}{y_t^2}\cot\theta\ dt-\frac{\sin^2\theta}{y_t}\ dB_t,\notag
\end{align}
here $U_t=-B_t$, and
\[
\partial_t f_t=\frac{\sin^2\theta}{y_t^2}(1+f_t^2)\cdot\frac{a}{f_t+\cot\theta_t}.
\]
We now change time to $s=s(t)$ such that $ds=\sin^4\theta_t/y_t^2\ dt$. Taking $\kappa=8/3 (a=3/4)$, this leads to the equations
\begin{equation}\label{E:theta}
d\theta_t=-\frac{1}{2}\cot\theta_t\ dt-dB_t=\frac{1}{4}(\tan\theta/2-\cot\theta/2)\ dt-dB_t.
\end{equation}
and
\begin{equation}\label{E:CLi}
\partial_t f_t(z)=\frac{1+f_t(z)^2}{\sin^2\theta_t}\cdot\frac{3/4}{f_t(z)+\cot\theta_t},\quad f_0(z)=z,
\end{equation}
where, by a  slight abuse of notation, we used the same symbols for the time-changed processes. This time-change is familiar. Whereas equation \eqref{E:CL} corresponds to parametrizing the curve $\gamma$ by {\em half-plane capacity from infnity}, equation \eqref{E:CLi} corresponds to parametrizing by {\em conformal radius from i}, see \cite{lawlerPC}. Indeed,
\[
\Upsilon_t\equiv|f_t'(i)|^{-1}=\frac{\Im g_t(i)}{|g_t'(i)|}
\]
is the conformal radius of $\mathbb H\backslash\gamma(0,t]$ from $i$ in the following sense. Denote $h$ the homography
\begin{equation}\label{E:homo}
w\mapsto h(w)=i(1+w)/(1-w),
\end{equation}
mapping the unit disk $\mathbb U$ onto $\mathbb H$ and sending $0$ to $i$. Define $\Phi$ by $\Phi(w)=\Upsilon_t\cdot(h^{-1}\circ f_t\circ h)(w)$. Then $\Phi$ maps the slit disk $h^{-1}(\mathbb H\backslash\gamma(0,t])\subset\mathbb U$ conformally onto $\{|z|<\Upsilon_t\}$, so that $\Phi(0)=0$ and $\Phi'(0)=1$. On the other hand, it follows from \eqref{E:CLi} that $\Upsilon_t$ satisfies the ODE
\[
\dot{\Upsilon}_t=-\frac{3}{2}\Upsilon_t,\quad\Upsilon_0=1,
\]
whence
\begin{equation}\label{E:CR}
\Upsilon_t=e^{-\frac{3}{2}t}.
\end{equation}

We also note that the time-changed process $\theta$ is a Legendre process of index $\nu=-1$, see \cite{FLM}. It behaves at the boundary points $0$ and $\pi$ like a Bessel process of dimension 0. In particular, $\theta$ is absorbed once it reaches the boundary and the time of absorption is the conformal radius in $i$ of the component of $\mathbb H\backslash \gamma(0,\infty)$ that contains $i$.

Using Girsanov's theorem we condition $\gamma$ to pass to the right of $i$. Note that $q(x)=\sin^2(\theta/2)$, if $x=-\cot\theta$. Thus, conditioning introduces an additional drift term $(\partial/\partial\theta)\ln\sin^2\theta/2=\cot\theta/2$ to the stochastic differential equation \eqref{E:theta}, giving
\begin{equation}\label{E:theta-cond}
d\theta_t=\frac{1}{4}\left(\tan\theta/2+3\cot\theta/2\right)dt-dB_t.
\end{equation}
The conditioned process is a {\em generalized Legendre process} of index $(\nu,\mu)=(1,-1)$, and its transition density can be written down explicitly using Jacobi polynomials, see \cite{FLM}. Based on the expression for the density it can be shown that the conditioned process is a Feller process, similar to the proof of the Feller property for Bessel processes in \cite{RY}. The boundary behavior of the conditioned process at $\theta=0$ is that of a 4-dimensional Bessel process, at $\theta=\pi$ it is that of a 0-dimensional Bessel process. In particular, if the conditioned process starts at $\theta_0=0$ and $\tau_\phi=\inf\{t:\theta_t=\phi\}$, $\phi\in(0,\pi]$, then
\begin{equation}\label{E:integrable}
\int_0^{\tau_\phi}\sin^{-2}(\theta_t)\ dt<\infty,\quad a.s.,
\end{equation}
for any $\phi\in(0,\pi)$, see \cite{lawlerbook}.

\begin{theorem}\label{T:bbubble}
Let $\theta_t$ be a process that satisfies the SDE \eqref{E:theta-cond} with $\theta_0=0$ on some probability space $(\Omega,\mathcal F,P)$. Then
the solution $f_t$ to the equation \eqref{E:CLi} for the driving function $\theta_t$ exists on the interval $[0,\tau_\pi)$. Furthermore, there is a simple curve $\gamma:(0,\tau_{\pi})\to\mathbb H$ with
\[
\lim_{t\searrow0}\gamma_t=\lim_{t\nearrow\tau_{\pi}}\gamma_t=\infty,
\]
such that $f_t$ maps $\mathbb H\backslash\gamma(0,t]$ conformally onto $\mathbb H$. If we set $\gamma_0=\infty$, then $\gamma[0,\tau_{\pi})$ is the boundary of a Jordan domain containing $i$. Finally, if $D\subset\mathbb H$ is a simply connected domain that contains $i$ and a boundary neighborhood of $\infty$ in $\mathbb H$, then
\begin{equation}\label{E:bubble1}
P(\gamma\subset D)=\Phi_D'(\infty)^2,
\end{equation}
where $\Phi_D$ is the conformal map from $D$ onto $\mathbb H$ fixing $i$ and $\infty$, and if $D,D'$ are ring domains in $\mathbb H$ of the same modulus containing boundary neighborhoods of $\infty$ in $\mathbb H$ and so that $i\notin D, i\notin D'$, then
\begin{equation}\label{E:bubble2}
P(\gamma\subset D)=\Phi_{D,D'}'(\infty)^2 P(\gamma\subset D'),
\end{equation}
where $\Phi_{D,D'}$ is the conformal map from $D$ onto $D'$ fixing $\infty$.
\end{theorem}

\begin{proof}
The existence of the solution $(f_s:s\in[0,t])$ to \eqref{E:CLi} for $t\in[0,\tau_{\pi})$ follows from \eqref{E:integrable}. Just as for the chordal Loewner equation \eqref{E:CL}, there is a growing, relatively closed set $K_t$ such that $f_t$ maps $\mathbb H\backslash K_t$ conformally onto $\mathbb H$. It is also clear from \eqref{E:CLi} that $f_t(i)=i$. For $\phi\in(0,\pi)$, write
\[
K_{\tau_{\phi}}=K_{\delta\wedge\tau_{\phi}}\cup f_{\delta\wedge\tau_{\phi}}^{-1}(f_{\delta\wedge\tau_{\phi}}(K_{\tau_{\phi}}\backslash K_{\delta\wedge\tau_{\phi}})).
\]
It follows from the flow property that if $\tilde{K}_{\tau_{\phi}-\delta\wedge\tau_{\phi}}\equiv f_{\delta\wedge\tau_{\phi}}(K_{\tau_{\pi}}\backslash K_{\delta\wedge\tau_{\phi}})$, and $\tilde{f}_t$ is the solution of \eqref{E:CLi} with driving function $\tilde{\theta}_t=\theta_{t+\delta\wedge\tau_{\phi}}$, then $\tilde{f}_{\tau_{\phi}-\delta\wedge\tau_{\phi}}$ is the conformal map from $\mathbb H\backslash\tilde{K}_{\tau_{\phi}-\delta\wedge\tau_{\phi}}$ onto $\mathbb H$, that leaves $i$ and $\infty$ fixed. From the Markov property of $\theta_t$ it follows that $\tilde{\theta}_t$ is the driving function of an $\text{SLE}_{8/3}$, started at $x=-\cot\tilde{\theta}_0$, conditioned to pass right of $i$. Hence, w.p. 1, $\tilde{K}_{\tau_{\phi}-\delta\wedge\tau_{\phi}}$ is given by a simple curve. Then $K_{\tau_{\phi}}\backslash K_{\delta\wedge\tau_{\phi}}$, as the conformal image of a simple curve, is given by a simple curve, and, letting $\delta\to0$, we get, w.p.1, $K_{\tau_\phi}$ is given by a simple curve.

We denote this curve by $\gamma$. To prove \eqref{E:bubble1}, consider the bounded martingale
\begin{align}
M_{\delta}&\equiv P(\gamma\subset D|\gamma[0,\delta\wedge\tau_\phi])\notag\\
&=1\{\gamma(0,\delta\wedge\tau_\phi)\subset D\}\mathbb E^P[1\{\gamma[\delta\wedge\tau_\phi,\tau_\pi)\subset D\}|\gamma[0,\delta\wedge\tau_\phi]].\notag
\end{align}
Then, a.s., $M_{\delta}\to M_0=P(\gamma\subset D)$ as $\delta\to0$. On the other hand, applying the flow and Markov property as in the preceding paragraph, we have
\begin{align}
&\mathbb E^P[1\{\gamma[\delta\wedge\tau_\phi,\tau_\pi)\subset D\}|\gamma[0,\delta\wedge\tau_\phi]]\notag\\
&=P_{x_{\delta}}(\gamma\subset f_{\delta\wedge\tau_\phi}(D))\notag\\
&=(\Psi_\delta'(x_\delta)\Psi_\delta'(\infty))^{5/8}q(\Psi_\delta(x_{\delta}))/q(x_\delta),\notag
\end{align}
where $x_{\delta}=-\cot\theta_{\delta\wedge\tau_\phi}$ and $\Psi_\delta$ is the conformal map from $f_{\delta\wedge\tau_\phi}(D)$ onto $\mathbb H$ fixing $i$ and $\infty$. It is straightforward to see that $(\Psi_\delta,\delta>0)$ is a normal family of conformal maps and $\Psi_\delta\to\Phi_D$ as $\delta\to0$ in the sense of Cartheodory convergence. It then follows that
\[
\Psi_\delta'(x_\delta)\Psi_\delta'(\infty)\to1,\quad q(\Psi_\delta(x_{\delta}))/q(x_\delta)\to\Phi_D'(\infty)^2
\]
as $\delta\to0$. Since, a.s., $1\{\gamma(0,\delta\wedge\tau_\phi)\subset D\}\to1$ as $\delta\to0$, equation \eqref{E:bubble1} now follows.

The proof for \eqref{E:bubble2} is analogous and is omitted.
\end{proof}

We call the curve $\gamma$, under $P$, a {\em boundary bubble attached at $\infty$ and encircling $i$}, and denote its distribution by $Bub^{(i)}$. Equations \eqref{E:bubble1} and \eqref{E:bubble2} then say that $Bub^{(i)}$ is a measure on boundary bubbles which transforms like a quadratic differential.

\begin{remark}
Using Schramm's formula for chordal $\text{SLE}_\kappa$ we can construct in the same way a bubble measure for other values of $\kappa\le4$. However, for $\kappa\neq8/3$ this measure does not transform like a quadratic differential. The method we employ in the next section to construct a loop measure from a bubble measure thus does not extend to other values of $\kappa$.
\end{remark}

\section{loop measure}\label{S:E}

Denote $U$ a ring domain with boundary components $C_1,C_2$. A {\em cross-cut of $U$ from $C_1$ to $C_2$} is a homeomorphism $\gamma:(0,T)\to U$ for which $\gamma_0\equiv \lim_{t\searrow0}\gamma_t$ and $\gamma_T\equiv\lim_{t\nearrow T}\gamma_t$ exist and belong to $C_1$ and $C_2$, respectively. Then $\mod(U\backslash\gamma[0,t])$ is a continuous, strictly decreasing function of $t$, taking the value $a=\mod(U)$ for $t=0$, and $0$ for $t=T$. In particular, we may parameterize the cross-cut by conformal modulus, so that
\[
\mod(U\backslash\gamma[0,t])=a-t,\quad t\in[0,a].
\]

For $a\ge0$, let $\rho=e^a$ and set
\[
U_a=\mathbb H\backslash\{z:\left|z-i\frac{\rho^2+1}{\rho^2-1}\right|\le\frac{2\rho}{\rho^2-1}\}.
\]
Then $z\mapsto(i+z)/(i-z)$ maps $U_a$ conformally onto $\{1<|z|<\rho\}$, so that $\mod(U_a)=a$. Let $\gamma$ be a cross-cut of $U_a$ starting from $\infty$, parameterized by conformal modulus. Let $\varphi_{t,a}$ be the conformal map from $U_a\backslash\gamma[0,t]$ onto $U_{a-t}$ so that $\varphi_{t,a}(\gamma_t)=\infty$. We define a measure $\mu_a$ by
\begin{equation}\label{E:defmu}
\mu_a=\int_0^a \varphi_{t,a}^* (Bub^{(i)}\upharpoonright\{\ell\subset U_{a-t}\cup\{\infty\}\})\ dt.
\end{equation}

\begin{theorem}\label{T:WM}[Werner measure]
If $U$ is a ring domain in $U_a$ so that $\langle U\rangle\subset\langle U_a\rangle$, then
\begin{equation}\label{E:main}
\mu_a(\langle U\rangle)=\int_0^{\mod(U)}Bub^{(i)}(\ell\subset U_b\cup\{\infty\})\ db.
\end{equation}
In particular, the definition of $\mu_a$ is independent of the cross-cut $\gamma$ used in \eqref{E:defmu}, $\mu_a(\langle U\rangle)=\mu_{a'}(\langle U\rangle)$ for any $a'> a$, and there exists a unique conformal restriction family $\{\mu^S\}$ with $\mu^{\langle U_a\rangle}=\mu_a$ for each $a>0$.
\end{theorem}

\begin{proof}
Let $\tilde{U}_t=\varphi_t(U\backslash\gamma(0,t])$. Then, by definition,
\begin{equation}\label{E:int}
\mu_a(\langle U\rangle)=\int_0^a Bub^{(i)}(\ell\subset\tilde{U}_t\cup\{\infty\})\ dt.
\end{equation}
Note that if $U$ is strictly contained in $U_a$, then the integrand is nonzero only for a part of the interval of integration. In fact, without changing the value of the integral we may restrict the integration to those $t$ for which $\gamma_t\in U$ and there exists a loop $\ell\in\langle U\rangle$ so that $\gamma(0,t]\cap\ell=\emptyset$. This set of times $t$ consists of at most countably many open subintervals of $(0,a)$. If $(s,s')$ is such a subinterval, then $\gamma(s,s')$ is an open Jordan arc in $U$ which begins and ends on the boundary of $U$. Except for the last (with respect to the natural time ordering) of these Jordan arcs they all begin and end on the same boundary component. The last Jordan arc is a crossing of $U$.

Suppose now that $t$ is contained in one of these subintervals, say $(s,s')$.
If $b$ is the conformal modulus of $\tilde{U}_t$, denote $\psi$ the conformal map from $\tilde{U}_t$ onto $U_b$, fixing $\infty$. By \eqref{E:bubble2},
\begin{equation}\label{E:sub1}
Bub^{(i)}(\ell\subset\tilde{U}_t\cup\{\infty\})=Bub^{(i)}(\ell\subset U_b\cup\{\infty\})\ \psi'(\infty)^2.
\end{equation}
Furthermore, for $t\in(s,s')$, $b$ is a strictly decreasing function of $t$. In fact,
\begin{equation}\label{E:sub2}
db=-\psi'(\infty)^2\ dt,
\end{equation}
as follows from the Loewner equation in doubly connected domains, see \cite[Theorem 3.2]{bauerSLE83}. Equations \eqref{E:int}, \eqref{E:sub1}, and \eqref{E:sub2} now imply \eqref{E:main}. The two statements following \eqref{E:main} are immediate consequences of \eqref{E:main}, and the existence of a conformal restriction family follows from Proposition \ref{P:restriction}.
\end{proof}

\begin{remark}
As the proof shows, it is the fact that the bubble measure transform like a quadratic differential which renders the integral in the definition of $\mu_a$ independent of the choice of cross-cut. A more conceptual argument is as follows: Quadratic differentials span the cotangent space of the moduli space, which is 1-dimensional for ring domains. The integration in the definition of $\mu_a$ is the pairing of a chain (the path in moduli space induced by the cross-cut) and a co-chain (the pull-back of the measure-valued quadratic differential given by the bubble measure). This pairing is well-defined and gives the same value (measure) for all chains in the same homology class. The cross-cuts of a ring domain induce the same chain in moduli space.
\end{remark}

\section{The conformal radius of boundary bubbles}\label{S:CR}

Because of \eqref{E:main} it would be of great interest to have an explicit formula for $f(a)\equiv Bub^{(i)}(\ell\subset U_a\cup\{\infty\})$. Based on physical arguments using the $O(n)$-model, Cardy conjectured a formula for the total mass $|\mu_a|$ from which the desired formula would follow by differentiation. On the other hand, different mathematical proofs establishing the asymptotics of $f$ as $a\searrow0$ have been given in \cite{wernerSAL} and  \cite{bauerSLE83}. The asymptotics these authors found agree with the asymptotics that would follow from Cardy's formula. In this section we compute the distribution of the conformal radius of boundary bubbles under the measure $Bub^{(i)}$ and deduce upper and lower bounds for $f(a)$ from it.

For a boundary bubble $\ell$ attached at infinity and encircling $i$ in $\mathbb H$, denote $I(\ell)$ the interior of the bubble, i.e. the component of $\mathbb H\backslash\ell$ containing $i$, and $r(\ell)$ the conformal radius of $I(\ell)$ from $i$.

\begin{lemma}\label{L:bubble}
Let $q=e^{-a}$. For a boundary bubble $\ell$ attached at infinity and encircling $i$ the following holds:
\begin{itemize}
\item If $r(\ell)\ge 4q$, then $\ell\subset U_a\cup\{\infty\}$.
\item If $\ell\subset U_a\cup\{\infty\}$, then $r(\ell)\ge q$.
\end{itemize}
\end{lemma}

\begin{proof}
Denote $f$ the unique conformal map from $I(\ell)$ onto $\mathbb H$ keeping $i$ and $\infty$ fixed. Then $|f'(i)|(h^{-1}\circ f^{-1}\circ h)$ is conformal on the unit disk, fixes $0$, and has derivative 1 there. Thus, by the Koebe 1/4 Theorem, the image contains the disk $\{|z|<1/4\}$. Since $r(\ell)=|f'(i)|^{-1}$, it follows that
\[
h^{-1}(I(\ell))\supset\{|z|<r(\ell)/4\}.
\]
Whence $r(\ell)\ge4q$ implies $I(\ell)\supset h(\{|z|<q\})$ and we get the first statement.

The second statement is a consequence of the monotonicity of the conformal radius: If $I(\ell)\subset I(\ell')$, then $r(\ell)\le r(\ell')$.
\end{proof}

\begin{theorem}\label{T:bubble}
For $q\in (0,1)$, we have
\[
Bub^{(i)}(\{\ell: r(\ell)\ge q\})=\prod_{n=1}^{\infty}\left(1-q^{\frac{2n}{3}}\right)^3.
\]
\end{theorem}

\begin{proof}
Consider the infinitesimal generator $L$ of the conditioned process $\theta_t$,
\[
L=\frac{1}{2}\frac{d^2}{d\theta^2}+\left[\frac{3}{4}\cot\frac{\theta}{2}+\frac{1}{4}\tan\frac{\theta}{2}\right]\frac{d}{d\theta}.
\]
Its eigenfunctions and eigenvalues are
\[
p_n(\theta)=P_n^{(1,-1)}(\cos\theta),\quad \lambda_n=-\frac{1}{2}n(n+1),\quad n=0,1,2,\dots
\]
where $P_n^{(\nu,\mu)}$ denotes the $n$-th Jacobi polynomial of index $(\nu,\mu)$, given, for example, by Rodrigues' formula
\begin{equation}\label{E:Rodrigues}
(1-x)^{\nu}(1+x)^{\mu}P_n^{(\nu,\mu)}(x)=\frac{(-1)^n}{2^n n!}\frac{d^n}{dx^n}[(1-x)^{n+\nu}(1+x)^{n+\mu}],
\end{equation}
see \cite[p. 99]{specialfunctions}. The eigenfunctions for the adjoint operator $L^*$ are given by $p_n^*(\theta)=p_n(\theta)2\sin^3\frac{\theta}{2}\cos^{-1}\frac{\theta}{2}$, and
\[
\int_0^\pi p_n(\theta)p_m^*(\theta)\ d\theta=\frac{2n+2}{2n^2+n}\delta_{mn}.
\]
It follows readily from \eqref{E:Rodrigues} that $p_0(\pi)= 1$ (in fact, $p_0\equiv1$), while $p_n(\pi)=0$ for all $n\ge1$. Since the conditioned process is absorbed at $\theta=\pi$, it follows that the transitiondensity $p_t(\theta,\theta')$ of the conditional process to go from $\theta$ at time zero to $\theta'$ at time $t$, does not involve $p_0$ and is given by
\begin{equation}\label{E:transitiondensity}
p_t(\theta,\theta')=\sum_{n=1}^\infty e^{-t n(n+1)/2}\ \frac{2n^2+n}{2n+2}\ p_n(\theta)p_n^*(\theta').
\end{equation}
From the definition of $Bub^{(i)}$ from $P$ and \eqref{E:CR} it follows that
\begin{equation}\label{E:BubitoP}
Bub^{(i)}(\{\ell:r(\ell)\ge q\})=P\left(\exp\left(-\frac{3}{2}\tau_\pi\right)\ge q\right),
\end{equation}
where, we recall, $\tau_\pi$ is the lifetime of the conditioned process $\theta_t$. As the conditioned process starts at $\theta=0$, we need to calculate
\[
P(\tau_\pi>t)=\int_0^\pi p_t(0,\theta')\ d\theta'.
\]
From \eqref{E:Rodrigues} we get $p_n(0)=P_n^{(1,-1)}(1)=n+1$ as well as
\begin{align}
\int_0^\pi p_n^*(\theta')\ d\theta'&=\int_{-1}^1 P_n^{(1,-1)}(x)\frac{1-x}{1+x}\ dx\notag\\
&=\int_{-1}^1\frac{(-1)^n}{2^n n!}\frac{d^n}{dx^n}[(1-x)^{n+1}(1+x)^{n-1}]\ dx\notag\\
&=(-1)^{n+1}\frac{2}{n}.
\end{align}
Thus
\begin{align}\label{E:triple}
P(\tau_\pi>t)&=\sum_{n=1}^\infty (-1)^{n-1}(2n+1)e^{-t n(n+1)/2}\notag\\
&=1-\prod_{n=1}^\infty(1-e^{-t})^3,
\end{align}
where the last equality follows from Jacobi's triple product formula, see \cite[(10.4.9)]{specialfunctions}. Equations \eqref{E:BubitoP} and \eqref{E:triple} now imply the theorem.
\end{proof}

\begin{remark}
Applying Jacobi's triple product formula we obtain an alternative expression for the distribution of the conformal radius,
\begin{equation}\label{E:qsmall}
Bub^{(i)}(\{\ell:r(\ell)\ge q\})=\sum_{n=0}^\infty(-1)^n(2n+1)q^{n(n+1)/3}.
\end{equation}
The Dedekind $\eta$-function is given by
\[
\eta(s)=q^{1/24}\prod_{n=1}^\infty(1-q^n),\quad |q|<1,
\]
where $q=e^{2\pi is}$. Thus we may also write
\begin{equation}\label{E:eta}
Bub^{(i)}(\{\ell:r(\ell)\ge q^{3/2}\})=q^{-1/8}\eta(s)^3.
\end{equation}
\end{remark}

\begin{corollary}
We have
\begin{equation}\label{E:smallq}
Bub^{(i)}(\{\ell:r(\ell)\le q\})\sim 3q^{2/3},\quad\text{as }q\searrow0,
\end{equation}
and
\begin{equation}\label{E:smalla}
Bub^{(i)}(\{\ell:r(\ell)\ge e^{-a}\})\sim \left(\frac{3\pi}{a}\right)^{3/2} \exp\left(-\frac{3\pi^2}{4a}\right),\quad\text{as }a\searrow0.
\end{equation}
\end{corollary}

\begin{proof}
Equation \eqref{E:smallq} follows immediately from \eqref{E:qsmall}. Next, recall that $\eta(-1/s)=\sqrt{\frac{s}{i}}\eta(s)$, see \cite[Theorem 10.12.8]{specialfunctions}. Thus, from \eqref{E:eta} we get
\begin{equation}\label{E:largeq}
Bub^{(i)}(\{\ell:r(\ell)\ge q^{3/2}\})=q^{-1/8}\left(\frac{3\pi}{a}\right)^{3/2}e^{-\frac{3\pi^2}{4a}}\prod_{n=1}^\infty\left(1-e^{-\frac{6\pi^2 n}{a}}\right)^3.
\end{equation}
\end{proof}

\begin{corollary}\label{C:bounds}
For $q=e^{-a}\in (0,1)$ we have
\begin{equation}
Bub^{(i)}(\ell\subset U_a\cup\{\infty\})\le \prod_{n=1}^{\infty}\left(1-q^{\frac{2n}{3}}\right)^3,
\end{equation}
and for $q=e^{-a}\in(0,1/
4)$, we have
\begin{equation}
\prod_{n=1}^{\infty}\left(1-(4q)^{\frac{2n}{3}}\right)^3\le Bub^{(i)}(\ell\subset U_a\cup\{\infty\}).
\end{equation}
\end{corollary}

\begin{proof}
The bounds follow from Theorem \ref{T:bubble} and Lemma \ref{L:bubble}.
\end{proof}

\begin{remark}
We know from \cite[Lemma19]{wernerSAL} that
\[
Bub^{(i)}(\ell\subset U_a\cup\{\infty\})\sim \frac{c}{a^2}\exp\left(-\frac{5\pi^2}{4a}\right),\quad\text{as }a\searrow0
\]
for some constant $c>0$. Thus the probability for a boundary bubble to stay in a very ``thin'' annulus of modulus $a$ decays at a faster rate as $a\searrow0$ than the probability that a boundary bubble encircling $i$ incloses a region of conformal radius $e^{-a}$. A heuristic indication as to why this is so is the fact that the slit unit disk $\{|z|<1\}\backslash[\sqrt{\delta},1)$ for small $\delta$ has conformal radius from 0  of the order $1-O(\delta)$. Thus boundary bubbles in the unit disk of conformal radius $1-\delta$ can ``venture far outside'' the annulus $\{1-\delta<|z|<1\}$.
\end{remark}

\end{document}